\documentclass{amsart}
\usepackage{verbatim}
\usepackage{amsthm,amstext,amsmath,amscd,amssymb,latexsym}
\usepackage{color}
\usepackage[matrix,arrow]{xy}
\usepackage{amsfonts,enumerate,array}
\usepackage[colorlinks=true]{hyperref}
\usepackage{todonotes}
	\usepackage{verbatim}
\usepackage[toc,page]{appendix}
\usepackage{listings}

\usepackage{color}

\newtheorem{theorem}{Theorem}[section]
\newtheorem{lemma}[theorem]{Lemma}
\newtheorem{proposition}[theorem]{Proposition}
\newtheorem{corollary}[theorem]{Corollary}
\newtheorem{conjecture}[theorem]{Conjecture}

\theoremstyle{definition}
\newtheorem{definition}[theorem]{Definition}

\newtheorem{example}[theorem]{Example}
\theoremstyle{definition}
\newtheorem{remark}[theorem]{Remark}

\newcommand{\reg}{\rm{reg}}
\newcommand{\ls}{{\mathcal L}}
\newcommand{\PP}{{\mathbb P}}
\newcommand{\Ii}{\mathcal I}

\DeclareMathOperator{\vdim}{vdim}
\DeclareMathOperator{\edim}{edim}
\DeclareMathOperator{\ldim}{ldim}

\newcommand{\w}{\rm{w}}

\date{}

\title{On Segre's bound for fat points in $\mathbb {P}^n$}

\author{Edoardo Ballico}
\email{{\tt ballico@science.unitn.it}}
\address{Dept. of Mathematics, University of Trento, 38050 Povo (TN), Italy}

\author{Olivia Dumitrescu}
\email{{\tt  dumitrescu@math.uni-hannover.de}}
\address{Institut f\"ur Algebraische Geometrie GRK 1463, Welfengarten 1, 30167 Hannover, Germany}

\author{Elisa Postinghel}
\email{{\tt elisa.postinghel@wis.kuleuven.be}}
\address{KU Leuven, Department of Mathematics, Celestijnenlaan 200B, 3001 Heverlee,
Belgium}

\thanks{The first author was partially supported by MIUR and GNSAGA of INdAM (Italy). The second author is member of ``Simion Stoilow'' Institute of Mathematics of the Romanian Academy.
The third author is supported by the Research Foundation - Flanders (FWO)}

\keywords{Segre's bound, generalised Segre's bound, regularity index, fat points}

\subjclass[2010]{Primary: 13D40 Secondary: 14C20}

\begin{document}

\begin{abstract} For a scheme of fat points $Z$ defined by the saturated
 ideal $\mathcal{I}_Z$, the regularity index computes the Castelnuovo-Mumford
 regularity of the Cohen-Macaulay ring $R/\mathcal{I}_Z.$ 
For points in  ``general position'' we improve the bound for the regularity index 
computed by Segre for $\mathbb {P}^2$ and generalised by Catalisano, 
Trung and Valla for $\mathbb {P}^n$.
Moreover, we prove that the generalised Segre's bound conjectured
 by Fatabbi and Lorenzini holds for $n+3$ arbitrary points in $\mathbb {P}^n$.
 We propose a modification of Segre's 
conjecture for arbitrary points and we discuss some evidences. 
\end{abstract}

\maketitle

\section{Introduction}

Let $S=\{p_1,\dots ,p_s\}$ be a set of distinct points in $\mathbb {P}^n=\mathbb{P}^n_K$
 and let $\mathfrak{p}_1,\dots,\mathfrak{p}_s$ be the associated homogeneous 
prime ideals in the polynomial ring $R:=K[x_0,\dots,x_n]$, where $K$ is an 
algebraically closed field.
Given positive integers $m_1, \dots, m_s$ we denote by $Z:=\sum_{i=1}^s m_i p_i$ the
 $0$-dimensional subscheme of $\PP^n$ defined by the saturated ideal  
$\mathcal{I}_Z:=\mathfrak{p}_1^{m_1}\cap\cdots\cap\mathfrak{p}_s^{m_s}$.
We denote by $Z_{\textrm{red}}:=\sum_{i=1}^sp_i$ the support of $Z$ and 
by $\w(Z):=\sum_{i=1}^sm_i$ its weight.

Computing the value of the Hilbert function of $\mathcal{I}_Z$ at $d$ is equivalent to
 computing the dimension of the linear systems $\ls_{n,d}(m_1,\dots,m_s)$ of the 
degree-$d$ hypersurfaces of $\PP^n$ passing through each point $p_i$ with multiplicity
 at least $m_i$, for all $d\ge 0$.

The {\it regularity index}  $\reg(Z)$ of $Z$ is the smallest
positive integer $d$ such that $h^1(\PP^n,\mathcal {I}_Z(d)) =0$ or, equivalently, $h^1(\PP^n,\ls_{n,d}(m_1,\dots,m_s))=0$.
This number corresponds to the \emph{Castelnuovo-Mumford regularity} of the Cohen-Macaulay graded ring $R/\mathcal{I}_Z$.

\subsection{Segre's bound}
Let us assume, without loss of generality, that $m_1\ge\cdots\ge m_s\ge1$.

In 1961, Segre \cite{s} gave the following upper bound for the regularity index of a collection $Z$ of fat points in general position in $\PP^2$:
\begin{equation}\label{segre bound P2}
\reg(Z)\leq \max\left\{m_1+m_2-1,\left\lfloor \frac{w(Z)}{2}\right\rfloor\right\}.
\end{equation}

We must also mention that for plane points in general position, 
namely such that not three of them lie 
on a line, not six of them lie on a conic etc.,
 that was Segre's original hypothesis, the bound for the regularity index 
corresponds to the 
the famous conjecture of Segre, Harbourne, Gimigliano and Hirschowitz for 
linear systems of plane curves with fixed multiple base points.

In 1991, Catalisano \cite{c,c2} established that the bound \eqref{segre bound P2} holds
 sharp for sets of points that are three by three not collinear.
See \cite{g,g2,tv}
 for discussions about Segre's bound for fat points satisfying stronger conditions.

For arbitrary fat points in $\PP^2$, in 1969 Fulton \cite{f} gave the following upper bound:
\begin{equation}
\reg(Z)\le \w(Z)-1.
\end{equation}
It was proved to be sharp if and only if all points lie on a line by Davis and Geramita \cite{dg} in 1984.

Fatabbi in 1994 \cite{f0} proved that $\reg(Z)$ is bounded above by the maximum between the number $\left\lfloor w(Z)/2\right\rfloor$ and the maximal sum of the multiplicities of collinear points of $S$.

The above results were extended   to fat points in \emph{linearly general position} in $\PP^n$.  
Fix $n\ge2$, $s\ge2$. We say
the points $p_1,\dots ,p_s$ are in linearly general position in $\PP^n$ if for each integer $r\in \{1,\dots ,n-1\}$
we have $\sharp (S\cap L) \le r+1$,
for all $r$-dimensional linear subspaces $L\subset \mathbb {P}^n$.

Catalisano, Trung and Valla in \cite[Theorem 6]{ctv} showed  that if $Z$ is a collection of fat points in linearly general position in $\PP^n$, then
\begin{equation}\label{segre bound Pn}
\reg(Z)\leq \max\left\{m_1+m_2-1,\left\lfloor \frac{w(Z)+n-2}{n}\right\rfloor\right\}.
\end{equation}
 Moreover  they proved that the bound is sharp for 
$s\le n+2$ points in linearly general position and, for $s\ge n+3$,  when the points lie on a rational normal curve (\cite[Proposition 7]{ctv}). See also \cite{ceg,cg}.

The bound \eqref{segre bound Pn} is nowadays referred to as \emph{Segre's bound} for the regularity index of a collection of fat points $Z$ in $\PP^n$.

\subsection{Generalised Segre's bound}

For arbitrary fat points in $\PP^n$, Fatabbi and Lorenzini \cite{fl} gave the 
 following conjecture for the regularity index.

For any subset 
$L\subseteq \mathbb {P}^n$, write
$\w_L(Z)$ for the sum of all $m_p$, where $p\in S\cap L$ and $m_p$
 is the multiplicity of $Z$ at $p$. In particular $\w_{\PP^n}(Z)$ is the weight of $Z$, $\w(Z)$.

\begin{conjecture}\label{conjecture gen segre bound}
For $r=1,\dots,n$ and for any linear $r$-subspace $L$ of $\PP^n$, set
$$
T(Z,L):=\left\lfloor\frac{\w_L(Z)+r-2}{r}\right\rfloor.
$$
Then 
\begin{equation}\label{gen segre bound Pn}
\reg(Z)\le\max\{T(Z,L): L\subseteq \PP^n\}.
\end{equation}
\end{conjecture}

The bound in \eqref{gen segre bound Pn} is referred to as \emph{generalised Segre's bound} for the regularity index of an arbitrary collection of fat points in $\PP^n$. Notice that for schemes of fat points in linearly general position, the generalised Segre's bound \eqref{gen segre bound Pn} equals precisely Segre's bound \eqref{segre bound Pn}.
The case $n=2$ was proved to be true in \cite{f0}.

Conjecture \ref{conjecture gen segre bound} was established in the case $n=3$ by Thi\^en in 2000 \cite{t} and, independently, by Fatabbi and Lorenzini in 2001 \cite{fl}. The first author also proved the case of arbitrary double points in $\PP^4$.

More recently, Benedetti, Fatabbi and Lorenzini in \cite{bfl} proved that the conjecture holds for 
arbitrary $s\le n+2$ points of $\PP^n$.

Successively, Tu and Hung \cite{th} showed that Conjecture \ref{conjecture gen segre bound} holds for $n+3$ points that are \emph{almost equimultiple}, namely when $m_i\in\{m-1,m\}$, for all $i=1,\dots,n+3$.

In Section \ref{generalised segre's bound} we prove that Conjecture \ref{conjecture gen segre bound} holds for schemes with $n+3$ arbitrary fat points of $\PP^n$. 

In Section \ref{new bound}, Theorem \ref{proposition bdp bound}, that is based on the results of Brambilla, Dumitrescu and Postinghel \cite{bdp1}, 
improves \textit{Segre's bound} \eqref{segre bound Pn} 
(and also \eqref{gen segre bound Pn}) for  fat points in   
general position in $\PP^n$.
An instance of this is the scheme of seven double points in $\PP^3.$ 
In this case, Segre's bound \eqref{segre bound Pn} is $5$, but $\ls_{3,4}(2^7)$ 
has vanishing first cohomology group, as predicted by the bound \eqref{bdp bound} given in Theorem \ref{proposition bdp bound}, that is $4$. This is a well-known example that follows from the Alexander-Hirschowitz theorem \cite{AlHi}.

In Section \ref{Conjecture} we pose a modification of the Segre conjecture  
for the regularity index of a scheme of fat points, $\reg (Z)$, and prove it holds for $n=3$.

\subsection*{Acknowledgements} 
The second and third authors would like to thank the Research Center FBK-CIRM 
 Trento for the hospitality and financial support during their one month ``Research in Pairs'', Winter 2015. This project was initiated during this program.
We are grateful to the referee for his/her useful comments.

\section{Generalised Segre's bound for $n+3$ arbitrary points}\label{generalised segre's bound}

In this section we prove that Conjecture \ref{conjecture gen segre bound} is true for an arbitrary collection $Z$ of $n+3$ fat points in $\PP^n$. 
We recall that a \emph{non-degenerate} set of points in $\PP^n$ is one whose linear span is the whole space $\PP^n$.

\begin{theorem}\label{segre for arbitrary n+3}
Let $Z:=\sum_{i=1}^{n+3} m_i p_i$ be a scheme of fat points
supported on a non-degenerate set of distinct points in $\mathbb {P}^n$. Then 
$Z$ satisfies the generalised Segre's bound, namely
$$
\reg(Z)\le \max\{T(Z,L):L\subseteq\mathbb{P}^n\}.
$$
\end{theorem}

In order to prove this result, we need the following lemmas. 

\begin{lemma}\label{sssr1}
Fix an integer $a\ge 3$, hyperplanes $H, M\subset \mathbb {P}^4$, $H\ne M$, and sets $S_1\subset H\cap M$, $S_2\subset H\setminus (H\cap M)$, $S_3\subset M\setminus
(H\cap M)$ such that $\sharp (S_1) =3$, $\sharp (S_2)=\sharp (S_3) = 2$, $S_1\cup S_2$ is in linearly general position in $H$ and
$S_1\cup S_3$ is in linearly general position in $M$. Set $S:= S_1\cup S_2\cup S_3$ and $Z_{a}:= \sum _{p\in S} ap$. Then $h^1(\mathcal {I}_{Z_a}(2a-1))=0$.
\end{lemma}

\begin{proof}
Note that, up to a projective transformations, $S$ is uniquely determined.
The proof will be by induction on $a$. 

The case $a=3$ is an explicit computation, that can be easily performed with the help of a computer. If $e_0,\dots,e_4$ are the coordinate points of $\PP^4$, one
 can choose  $H$ to be the hyperplane spanned by the points $\{e_0,e_1,e_2,e_3\}$, $M$ to be the hyperplane spanned by the points $\{e_0,e_1,e_2,e_4\}$, $S_1=\{e_0,e_1,e_2\}$, $S_2=\{e_3, e_0+e_1+e_2+e_3\}$ and $S_3=\{e_4,e_0+e_1-e_2+e_4\}$. In this example it is easy to check that $h^1(I_{Z_3}(5))=0$ (see Section \ref{app1}).
 
Now assume $a>3$. We have $h^1(\mathcal {I}_S(2)) =0$. One can easily check this by studying the Castelnuovo residual sequence with respect to $H$.  Now we check
that $\mathcal {I}_S(2)$ is spanned. Using a residual exact sequence with the quadric hypersurface $H\cup M$ we get
$h^1(\mathcal {I}_{S\cup \{o\}}(2))=0$ for all $o\notin H\cup M$, i.e. $\mathcal {I}_S(2)$ is spanned outside
$H\cup M$. Then using a residual exact sequence with respect to $H$ (resp. $M$) we see that $\mathcal {I}_S(2)$ is spanned at each point of $H\setminus (S\cap H)$
(resp. $M\setminus (S\cap M)$). Now fix $o\in S$, say $o\in S\cap H$. Using the residual exact sequence of $H$ we get
$h^1(\mathcal {I}_{S\setminus \{o\}\cup 2o}(2)) =0$ and hence $\mathcal {I}_S(2)$ is globally generated at $o$. Since $\mathcal {I}_S(2)$ is spanned
and $S$ is finite, Bertini's theorem gives the existence of smooth quadric hypersurfaces $Q, Q, Q'\in |\mathcal {I}_S(2)|$ such that $Q\cap Q'$ is a smooth surface
and $C:= Q\cap Q'\cap Q''$ is a smooth curve. By Lefschetz' Theorem, $C$ is irreducible. By the adjunction formula $C$ is a canonically embedded smooth curve of genus $5$.
The inductive assumption gives $h^1(\mathcal {I}_{Z_{a-1}}(2a-3)) =0$ and so
$h^1(Q,\mathcal {I}_{Z_{a-1}\cap Q}(2a-3)) =0$ and $h^1(Q,\mathcal {I}_{Z_{a-1}\cap Q\cap Q'}(2a-3)) =0$. The residual exact sequence$$0 \to \mathcal {I}_{Z_{a-1}}(2a-3) \to \mathcal {I}_{Z_a}(2a-1)\to \mathcal {I}_{Z_a\cap Q}(2a-1) \to 0$$shows that it is sufficient to prove that $h^1(Q,\mathcal {I}_{Z_a\cap Q}(2a-1)) =0$. The residual exact sequence$$0 \to \mathcal {I}_{Z_{a-1}\cap Q, Q}(2a-3) \to \mathcal {I}_{Z_a\cap Q}(2a-1)\to \mathcal {I}_{Z_a\cap Q\cap Q'}(2a-1)\to 0$$shows that it is sufficient to prove that $h^1(Q\cap Q',\mathcal {I}_{Z_a\cap Q\cap Q',Q\cap Q'}(2a-1)) =0$. The residual exact sequence$$0 \to \mathcal {I}_{Z_{a-1}\cap Q\cap Q', Q\cap Q'}(2a-3) \to \mathcal {I}_{Z_a\cap Q\cap Q'}(2a-1)\to \mathcal {I}_{Z_a\cap C,C}(2a-1)\to 0$$shows that it is sufficient to prove that $h^1(C,\mathcal {I}_{Z_a\cap C,C}(2a-1)) =0$. Since $C$ is a complete intersection,
it is projectively normal. Thus it is sufficient to prove that $h^1(C,R)=0$, where $R$ is the line bundle $\mathcal {O}_C(2a-1)(-Z_a\cap C)$.
We have $h^1(C,R)=0$, because $C$ has genus $5$ and the Euler characteristic of $R$ is $\chi({R}) = 8(2a-1) -7a \ge 9$.
\end{proof}

\begin{remark}
Lemma \ref{sssr1} is false for $a=2$, even if we take $S$ to be a set of general points in $\mathbb {P}^4$ (it is an exceptional case in the list of Alexander-Hirschowitz \cite{AlHi}.
See also \cite{ale-hirsch, Po}). Also note that for $a=2$
the Segre number is $4$, because $\lfloor \frac{7\cdot 2 +2}{4}\rfloor =4$. 
\end{remark}

\begin{lemma}\label{sssr2}
Fix integers $n\ge 2$, $m>0$, $a>0$ and $t\ge a+m-1$. Fix a finite set $S\subset \mathbb {P}^n$ and $o\in S$. Fix a hyperplane $H\subset \mathbb {P}^n$
such that $o\notin H$. Let $\ell : \mathbb {P}^n\setminus \{o\} \to \mathbb {P}^{n-1}$ be the linear projection from
$o$. Set $S':= S\setminus \{o\}$ and $S_1:= \ell (S')$. Assume that $\ell |S'$ is injective. For all $p\in S$ fix an integer $m_p\ge 0$ with the restriction
that $m_o=m$ and that $m_p\le a$ for each $p\in S'$. Set $Z:= \sum _{p\in S} m_pp$. For each
integer $x\ge 0$ set $W_x:= \sum _{p\in S'} \max \{0,m_p-x\}\ell ({p})$ (hence $W_x=\emptyset$ for all $x\ge a$).
Assume $h^1(H,\mathcal {I}_{W_x}(t-x)) =0$ for all $x=0,\dots ,a$. Then $h^1(\mathcal {I}_Z(t)) =0$. 
\end{lemma}

\begin{proof}
For each integer $x\ge 0$ set $Z_x:= mo+\sum _{p\in S'} \max \{m_p-x\}\ell ({p})$. We have $W_x= Z_x\cap H$ and $Z_x = mo$ for all $x>0$. Choose a system $x_0,\dots ,x_n$ of homogeneous coordinates such
that $H = \{x_0=0\}$ and $o = (1:0:\dots :0)$. For each $\lambda \in \mathbb {K}\setminus \{0\}$ let $h_\lambda : \mathbb {P}^n\to \mathbb {P}^n$ be the automorphism defined
by the formula $h_\lambda (x_0:x_1:\cdots :x_n) = (\lambda x_0:x_1:\cdots :x_n).$ We have $h_\lambda (o)=o$.
Since $h_\lambda$ is an automorphism of $\mathbb {P}^n$, $Z$ and $h_\lambda (Z)$ have the same Hilbert function. Since $Z_0$ is a flat limit of the family
$\{h_\lambda (Z)\}_{\lambda \ne 0}$, it is sufficient to prove that $h^1(\mathcal {I}_{Z_0}(t)) =0$. For each integer $x=0,\dots ,a$ there is a residual exact sequence$$0\to \mathcal {I}_{Z_{x+1}}(t-x-1)\to \mathcal {I}_{Z_x}(t-x) \to \mathcal {I}_{W_x,H}(t-x)\to 0.$$
We can conclude using the assumptions on $W_x$ and that $h^1(\mathcal {I}_{Z_a}(t-a)) = h^1(\mathcal {I}_{mo}(t-a)) =0$,
because $t\ge a+m-1$.
\end{proof}

\begin{lemma}\label{claim 1} Fix integer $t\ge 2$, $z>z_1\ge \cdots \ge z_t >0$ and set $\eta := z+z_1 +1$. Then $z+z_1+\cdots +z_t\le \frac{t+1}{2}\eta -t$
if $\eta$ is even and $z+z_1+\cdots +z_t\le \frac{t+1}{2}\eta -(t-1)/2$ if $\eta$ is odd.
\end{lemma}
\begin{proof}
It is sufficient to prove the statement
when $z_i=z_1$ for all $i$.
We fix $\eta$. If $\eta$ is even, the left hand side of the inequality is maximal
if $z = \eta/2 $ and $z_1=\eta /2 -1$ and in this case we have $z+tz_1 = (t+1)\eta /2 -t$. If $\eta$ is odd, then the left hand side of the inequality
is maximal if $z = (\eta +1)/2$ and $z_1=(\eta -1)/2$ and in this case we have $z+tz_1= (t+1)(\eta +1)/2 -t$. 
\end{proof}

\begin{proof}[Proof of Theorem \ref{segre for arbitrary n+3}]
We will use the notation $S:=Z_{\textrm{red}}$. We will denote by $\alpha$ the Segre bound for $\reg(Z)$.

The proof is by induction on $n$ and $\w(Z)$. The case $n=1$ is 
obvious by the cohomology of line bundles on $\PP^1$. Hence
we may assume $n\geq 2$.
Now assume that $\w(Z)$ is as minimal as possible, 
i.e. $\w (Z)=n+3$, i.e. $m_i=1$ for all $i$. Since $S$ spans $\PP^n$, 
we need to prove that $h^1(\Ii _S(3)) =0$ if there
is a line $L\subset \PP^n$ with $\sharp (S\cap L) =4$ and $h^1(\Ii _S(2)) =0$ 
if there is no such a line.
Let $H\subset \PP^n$ be a hyperplane such that $\sharp (S\cap H)$ is maximal. In particular $S\cap H$ spans $H$. Since $S$ spans $H$, we have $n+1 \le \sharp (S\cap H)\le n+2$. Therefore $h^1(\Ii _{S\setminus S\cap H}(1)) =0$. If $L\subset \PP^n$ is a line with $\sharp (S\cap L) \ge 4$, then $L\subset H$. The inductive
assumption gives that $h^1(H,\Ii _{S\cap H}(3)) =0$ and that $h^1(H,\Ii _{S\cap H}(2)) =0$ if there is no line $L$ with $\sharp (S\cap L)=4$.

Hence we may use induction on $\w (Z)$.

If $S$ is in linearly general position, the statement is a particular case of \cite[Theorem 6]{ctv}. From now on we will assume that the points are not in linearly general position. In particular we will handle separately the two following cases.

\quad Case (1): $n+2$ points of $S$ are contained in a hyperplane. 

\quad Case (2): $n+1$ points of $S$ are contained in a hyperplane, but no hyperplane contains $n+2$ points.

\medskip
\quad Case (1).

After relabelling the points if necessary, we may assume that $H\cong \mathbb {P}^{n-1}$ is a hyperplane such that  $p_1,\dots,p_{n+2}\in H$ and $p_{n+3}\notin H$. Set
$W:= \textrm{Res}_H(Z)=\sum_{i=1}^{n+2}(m_i-1)p_1+m_{n+3}p_{n+3}$.  
Let $\beta$ be the Segre bound
for $\reg(W)$. 
 Consider the residual exact sequence 
 \begin{equation}\label{restriction}
0 \to \mathcal{I} _W(d-1) \to \mathcal{I} _Z(d)\to \mathcal{I} _{Z\cap H,H}(d) \to 0.
\end{equation} 
It is sufficient to prove that $\reg(W) \le \alpha -1$ and $\reg(Z\cap H)\le \alpha$.
The latter
is obvious, because by the inductive assumption on $n$ and the fact that the linear subspaces arising in the test of the bound for $Z\cap H$ are some of the ones used in the definition of $\alpha$.
By the inductive assumption on $\w(Z)$, we may assume that $W$ 
satisfies the statement, namely that $\reg(W)\le\beta$.
Therefore it suffices to prove that $\beta \le \alpha -1$. 
For any linear space $M\subseteq \mathbb {P}^n$ we have $\w_M(W) = \w_M(Z) -\sharp (S\cap M\cap H)$.

Let $L\subseteq \PP^n$ be a linear $r$-subspace evincing $\beta$, i.e. $\beta=T(W,L)$. It is a subspace spanned by the points of $W_{\mathrm{red}}$ and
in particular by points of $S$. 
Since $L\subset\PP^n$, the definitions of $\alpha$ and of $\w_L$ give $\w_L(Z)\le r\alpha+1$.
Since $\sharp (S\cap H) =\sharp (S)-1$, we get $\w_L(W) 
\le \w_L(Z)-r\le r\alpha +1 -r$ and hence $\beta \le \alpha -1$.

\medskip

\quad Case (2). 
We may assume that $H$ is a hyperplane of $\mathbb{P}^n$ with $p_1\dots,p_{n+1}\in H$ and $p_{n+2},p_{n+3}\notin H$.
Similarly to Case (1), set
$W:= \textrm{Res}_H(Z)=\sum_{i=1}^{n+1}(m_i-1)p_i+m_{n+2}p_{n+2}+m_{n+3}p_{n+3}$.  
By the inductive assumptions and the residual exact sequence \eqref{restriction} it suffices to prove that $\beta \le \alpha -1$.

Let $L\subseteq \PP^n$ be a linear $r$-subspace evincing $\beta$. It is spanned by points in the support of $W$ and
in particular it is spanned by points of $S$. We have $\beta = \lfloor \frac{\w_L(W)+r-2}{r}\rfloor$ and $\w_L(W) = \w_L(Z) -\sharp (S\cap L\cap H)$.
Since $\w_L(Z) \le r\alpha +1$, we have $\beta \le \alpha -1$, unless $\sharp (S\cap L\cap H)\le r-1$, i.e. unless $L$ contains $p_{n+2}$ and $p_{n+3}$.
Assume that this is the case. We have $\sharp (S\cap L\cap H) =r-1$.

We consider the following cases.

\quad Case (a) $L$ has dimension $r\ge2$.

\quad Case (b) $L$ has dimension $1$.

\medskip 

\quad Case (a).

 In this case $\dim(L)\ge 2$, we claim that  $\w_L(W)\le 2(r+1)$. 
If $\w_L(W)=2(r+1)$ then $\beta=3$ and, moreover, $m_{n+2}+m_{n+3}\le4$, 
as $m_{n+2}+m_{n+3}-1\le\beta$, by the definition of $\beta$. Hence we can conclude 
that $\alpha\ge \beta+1=4$. Indeed,  let  $L'\subset L\cap H$ be the linear space 
spanned by $S\cap L \setminus \{p_{n+2},p_{n+3}\}$; notice that $\dim(L')=r-2$. 
Then we have
$\alpha\ge T(Z,L')=\lfloor (\w_{L'}(Z)+r-4)/(r-2)\rfloor =
\lfloor (\w_L(W)+(r-1)-m_{n+2}-m_{n+3}+r-4)/(r-2)\rfloor\ge
\lfloor (\w_L(W)+(r-1)-4+r-4)/(r-2)\rfloor=4$.
If $\w_L(W)< 2(r+1)$, then $\beta<3$. In this case we have $m_{n+2}+m_{n+3}\le 3$, hence $\max\{m_{n+2},m_{n+3}\}\le2$. Moreover, since $r\ge2$, it must be $m'_i\ge1$ for some $p_i\in S\cap L\cap H$, $i\ne n+2,n+3$. Therefore $m_i\ge2$ and this implies that $\alpha\ge m_i+\max\{m_{n+2},m_{n+3}\}-1\ge3\ge\beta+1$ and we conclude.

We are left with proving the claim. Let $I$ be the index set parametrizing the union of points $S\cap L$. Set $m'_i=m_i-1$, for all $i\in I\setminus\{n+2,n+3\}$, and $m'_{n+2}=m_{n+2}$, $m'_{n+3}=m_{n+3}$, so that $\w_L(W)=\sum_{i\in I}m'_i$.
The definition of $\beta$ implies that 
$$\frac{\w_L(W)+r-2}{r}\ge m'_i+m'_j-1,$$ for all $i,j\in I$, $i\ne j$.
If $r=2\rho-1$ ($\rho\ge2$), then 
\begin{align*}
\rho\frac{\w_L(W)+r-2}{r}& \ge (m'_1+m'_2-1)+\cdots+(m'_{n+2}+m'_{n+3}-1)\\
& =\w_L(W)-\rho.
\end{align*} 
One can easily check that this is equivalent to $\w_L(W)\le 2(r+1).$
If instead $r=2\rho$ ($\rho\ge1$), by a similar computation one obtains 
$$
\rho\frac{\w_L(W)+r-2}{r}\ge \w_L(W)-\rho-m'_i,
$$
for all $i\in I$. We leave it to the reader to check that by taking the sum over $i\in I$ of the above expressions, one  concludes that
 $\w_L(W)\le 2(r+1)$ also in this case.

\medskip

\quad Case (b).

 Assume that $L$ is the line spanned by $p_{n+2},p_{n+3}$ and that $\sharp(S\cap L)=2$. If $\alpha$  is not attained by the line $L$, then $\alpha>m_{n+2}+m_{n+3}-1=\beta$ and we conclude.
Assume now that $\alpha$ is attained by the line spanned by $p_{n+2},p_{n+3}$, i.e. $\alpha=m_{n+2}+m_{n+3}-1$.

Without loosing
generality we may assume $m_{n+3}\ge m_{n+2}$. The definition of $\alpha$ gives $m_i\le m_{n+2}$ for all $i\le n+1$. Hence, up to a permutation
of the first $n+1$ indices we may assume that the sequence $m_i$, $1\le i\le n+3$, is non-increasing. If $m_{n+3}>m_{n+2}$, then the lines $R$
with $\w_R(Z) =\alpha +1$ are spanned by $p_{n+3}$ and the points $p_i$ with $m_i =m_{n+2}$. In this case all such lines
contain $p_{n+3}$. If $m_{n+3}=m_{n+2}$ (and hence $\alpha$ is odd), then the lines $R$ with $\w_R(Z) =\alpha+1$
are the lines spanned by two points with multiplicity $m_{n+3} =(\alpha +1)/2$.

We will split the proof of the statement in the following cases.

\quad Case (b.1): homogeneous case, i.e. $m_1=...=m_{n+3}=m$.

\quad Case (b.2): there exists a subspace $N$ of $\dim(N)\le n-2$ containing $\dim(N)+2$ points of $S$.

\quad Case (b.3): every subspace $N$, with $\dim(N)\le n-2$, contains $\dim(N)+1$ points of $S$.
 
\medskip 

\quad Case (b.1).

If  $m=1$, we have $\alpha >1$. Since $S$ spans $\mathbb {P}^n$, we have
$\sharp (S\cap N)\le \dim (N)+3$ for all linear spaces $N\subsetneq \mathbb {P}^n$. Since $n\ge 4$ we see that $\alpha =3$ if $\sharp (S\cap R) = 4$ for some line $R$ and $\alpha =2$ in all other cases. 
Notice that the first case does not occur. Finally, if $\alpha =2$ the vanishing
of $h^1(\mathcal {I}_S(\alpha ))$ is well-known to hold. 

Assume $m\ge 2$, for all $i=1,\dots,n+3$, and $\alpha=2m-1$.
 For each integer $t=1,\dots ,n-1$ let
$\gamma _t$ be the maximal integer such that $\gamma _t(\alpha +1)/2 \le t\alpha +1$. We have $\gamma _t =t+1$ for $t=1,2$, $\gamma _3=5$ and $\gamma _t \ge t+3$
for all $t\ge 4$. Since $\alpha$ is the Segre bound of $Z$, we have $\sharp (S\cap N)\le \gamma _t$ for each $t$-dimensional linear space $N\subset \mathbb {P}^n$.
In particular $\sharp (S\cap N)\le 2$ for each line $N$ and $\sharp (S\cap N)\le 3$ for each plane $N$. By assumption, $\sharp (S\cap N)\le \dim (N)+2$
for each linear space $N\subsetneq \mathbb {P}^n$. Fix any hyperplane $M$ such that $\sharp (S\cap M) = n+1$ and set $S':= S\cap M$. Call $q,q'$ the points of $S\setminus (S\cap M)$
and $D$ the line spanned by $q$ and $q'$. Set $\{o\}:= D\cap M$ and $S'':= S'\cup \{o\}$. Since $\sharp (S\cap N)\le 2$ for each line $N$, we have $o\notin S'$. Set
$$Z':= \frac{\alpha +1}{2}o +\frac{\alpha +1}{2}q + \sum _{P\in S'} \frac{\alpha +1}{2}q.$$ Choose a system $x_0,\dots ,x_n$ of homogeneous coordinates such
that $M = \{x_0=0\}$ and $q = (1:0:\dots :0)$. For each $\lambda \in \mathbb {K}\setminus \{0\}$ let $h_\lambda : \mathbb {P}^n\to \mathbb {P}^n$ be the automorphism defined
by the formula $h_\lambda (x_0:x_1:\cdots :x_n) = (\lambda x_0:x_1:\cdots :x_n)$. We have $h_\lambda (q)=q$ and $h_\lambda ({p}) =p$ for each $p\in M$.
Since $h_\lambda$ is an automorphism, $Z$ and $h_\lambda (Z)$ have the same regularity. Since $Z'$ is a flat limit of the family $\{h_\lambda (Z)\}_{\lambda \ne 0}$,
the semicontinuity theorem for cohomology gives $\mathrm{reg}(Z')\ge \mathrm{reg}(Z)$. Hence is sufficient to
prove that $h^1(\mathcal {I}_{Z'}(\alpha ))=0$. Therefore we are done if $\alpha$ is Segre's bound for $Z'$. By the shape of the function $\gamma _t$ or by
the proof of Case 2 we see that it is sufficient to prove that $Z'\cap M$ has index of regularity $\alpha$. By the inductive assumption on $n$ and the shape
of the function $\gamma _t$ we see that it is sufficient to prove that $\sharp (S''\cap R) \le 2$ for each line $R \subset M$ and $\sharp (S''\cap N) \le 3$ for each plane $N\subset M$.
Assume the existence of a line $R \subset M$ such that $\sharp (S''\cap N) \ge 3$. Since $\sharp (S'\cap R)\le \gamma _1=2$, we have
$o\in R$ and $\sharp (S'\cap R)=2$. Hence $R\cup D$ spans a plane $A$ with $\sharp (S\cap A) =4>\gamma _2$, contradicting the assumption that
$\alpha$ is the Segre bound of $Z$. 

Now assume the existence of a plane $N$ such that $\sharp (S''\cap N) \ge 4$. Since $\gamma _2=3$ we see
that $o\in N$ and $\sharp (S'\cap N)=3$. Therefore $R\cup N$ spans a 3-dimensional linear subspace $A$ with $\sharp (S\cap A) =5$.
We repeat the construction taking a hyperplane $M(1)$ containing $A$ and spanned by points of $S$. Set $S'(1):= S\cap M(1)$,
$\{q(1),q(1)'\}:= S\setminus S'(1)$. Call $D(1)$ the line spanned by $\{q(1),q(1)'\}$ and set $o(1):= D(1)\cap M(1)$. We repeat the construction with $M(1)$
instead of $M$ and get $h^1(\mathcal {I}_Z(\alpha ))=0$, unless there is a 3-dimensional linear space $A(1)\subset \mathbb {P}^n$ with $\sharp (S\cap A(1)) = 5$ and a plane $N(1)\subset M(1)$ such that the linear span of $N(1)$ and $D(1)$ is $A(1)$ and $\sharp (S\cap N(1)) =3$.
 Set $b:= \sharp (S\cap A\cap A(1))$. By construction we have $b\le 3$. Let $u$ be the the dimension of the linear span $E$ of $A\cup A(1)$.
We have $\sharp (S\cap E) \ge 10-b$ and $u \le 7-b$. We get $E =\mathbb {P}^n$ and hence $n\le 7$. We also get that $S =S\cap (A\cup A(1))$, $S\cap A\cap A(1)$ is linearly independent
(it may be empty) and that $A\cap A(1)$ is spanned by $S\cap A\cap A(1)$.  For each $n$, any two sets $S', S'(1)$ with the properties just described are projectively equivalent.

\medskip

\quad Case (b.1.1).

 Assume $n=4$. If $\alpha =3$, then this case is excluded, because $n=4$, $s=7$, and $m_i=2$ for all $i$ has Segre number $4$. If $\alpha >3$, then we use the case $m_i = (\alpha +1)/2$ of Lemma \ref{sssr1}.

\medskip

\quad Case (b.1.2).

 Assume $n>4$. 
Consider first the case $n=5$, $m=2$. We want to prove that $h^1(\mathcal{I}_Z(3))=0$ for any union $Z$ of eight double points of $\PP^5$ such that $\sharp(S\cap N)\le \gamma_t$ for any $t$-dimensional subspace, with $\gamma_1=2$, $\gamma_2=3$, $\gamma_3=5$ and $\gamma_4=6$, the last inequality being sharp for a hyperplane $H$. It is enough to exhibit an example for which for every $t$ there is a subspace $N$ with $\sharp(S\cap N)=\gamma_t$ that satisfies the claim. One can show by computer 
that the statement holds for the set $S$ 
given by $e_0,\dots,e_5$, the coordinate points, and $e_0+e_1+e_2+e_3$, $e_0+e_1+e_5+e_6$ (see Section \ref{app2}).

Now assume $(n,m)\ne (5,2)$. With this restriction we know the vanishing
for the pair $(n-1,m)$ by the inductive assumption on $n$. Fix $o\in S\cap M$ and take a general hyperplane $N\subset \mathbb {P}^n$. Let $\ell : \mathbb {P}^n\setminus \{o\}\to N$
the linear projection from $o$. Set $S':= S\setminus \{o\}$ and $S_1:= \ell (S')$. For all integers $x\ge 0$ set 
$$W_x:= \left( \sum _{p\in S_1} \max \left\{\frac{\alpha +1}{2}-x,0\right\}p\right)\cap N.$$ Since $S_1$ is the configuration of $n+2$ points 
of $N$ corresponding to $\alpha$, we have $h^1(N,\mathcal {I}_{W_0}(\alpha ))=0$. Using Segre's bound in $N =\mathbb {P}^{n-1}$ we also get
the other vanishing needed in order to apply Lemma \ref{sssr2} with $m=(\alpha +1)/2$ and $t=\alpha$.

\medskip

\quad Case (b.2).

 Let $N\subset \mathbb {P}^n$ be a minimal subspace containing exactly $\dim (N)+2$ points of $S$. In this step we assume $y:= \dim (N) \le n-2$.
In this case we consider the residual sequence with respect to a hyperplane $H'$ containing $N$, spanned by points
of $S$ and containing $p_{n+3}$. We get $h^1(\mathcal {I}_Z(\alpha ))=0$, unless the two points, say $o_1$ and $o_2$, of $S\setminus (S\cap H')$ span a line $R$
with $S\cap R\cap H' =\emptyset$ and $m_{o_1}+m_{o_2} = \alpha +1$. 
This can not occur  if $m_{n+3}>m_{n+2}$. Indeed in this case each line $R$ with $T(R,Z) =\alpha +1$ contains $p_{n+3}$. Therefore
we may choose $H'$ to be a hyperplane containing $N\cup \{p_{n+3}\}$ and obtain $p_{n+3}\in S\cap R\cap H'$ for every line $R$.

Assume $m_{n+3} = m_{n+2}=\frac{\alpha+1}{2}$. The scheme $\widetilde{Z}:= \sum _{p\in S} \frac{\alpha +1}{2} p$ satisfies  the vanishing $h^1(\mathcal {I}_{\widetilde{Z}}(\alpha ))=0$ by Case (b1).
We conclude by noticing that $Z\subset \widetilde{Z}$ implies $\reg(Z)\le\reg(\widetilde{Z})$.

\medskip

\quad Case (b.3).

 In this case any $n$ points of $S$ are linearly independent.
 If $m_{n+2} =m_{n+3}=\frac{\alpha +1}{2}$, we set $\widetilde{Z}:= \sum _{p\in S} \frac{\alpha +1}{2} p$ and conclude by Case (b1) as above.

Assume $m_{n+3} > m_{n+2}$. Recall that since $m_{n+3}+m_{n+2}=\alpha +1$ and no triplet of  points of $S$ is supported on a line, each line $R$ with
$\w_R(Z)=\alpha +1$ is spanned by $p_{n+3}$ and a point with multiplicity $m_{n+2}$. Let $H'$ be the hyperplane spanned by
$p_{n+4},\dots,p_{4}$. Set $W':= \textrm{Res}_{H'}(Z)$. By the inductive assumption it is sufficient to prove that Segre's bound
for $W'$ is at most $\alpha -1$, i.e. that $\w_A({W'}) \le t(\alpha -1)+1$ for all integer $t=1,\dots ,n$ and all $t$-dimensional linear subspaces of $A \subseteq \mathbb {P}^n$.
It is sufficient to test the linear subspaces $A$ spanned by $S\cap A$.

If $t=n$, we prove the statement by noticing that 
 $\w(W')=\w(Z)-\sharp(S\cap H')\le n(\alpha-1)+1$. 

Now assume $t\le n-2$. By assumption $\sharp (S\cap A) = t+1$. We have $\sharp (S\cap A\cap H') =t+1 -\sharp (\{p_1,p_2,p_3\}\cap A)$ and
hence $\sharp (S\cap A\cap H')<t$ only if
at least two among $p_1,p_2,p_3$ are contained in $A$. First assume $\sharp (S\cap A\cap H') =t-1$. In this case we
have $\w_A(Z) \le m_{n+3}+\cdots + m_{n+5-t}+m_2+m_1$ and $\w_A(W')  \le m_{n+3}+\cdots + m_{n+5-t}+m_2+m_1 -(t-1)$. 
Lemma \ref{claim 1} with $\eta=\alpha$
gives $\w_A(Z) \le \frac{t+1}{2}\alpha -\frac{t-1}{2}$ and hence $\w_A({W'}) \le \frac{t+1}{2}\alpha -3\frac{t+1}{2}\le t(\alpha -1)+1$.
Now assume $\{p_1,p_2,p_3\}\subset A$ and hence $t\ge 2$. We get $\w_A({W'}) \le \frac{t+1}{2}\alpha -2t \le t(\alpha -1)+1$.

Now assume $t=n-1$ and $A\neq H'$. In the case $\sharp (S\cap A) =n$ we conclude as in the case $t\le n-2$.
Assume $\sharp (S\cap A) =n+1$. 
Obviously $\sharp(S\cap A\cap H')\ge n-1$, 
 hence $\w_A({W'}) \le \w_A(Z) -n+1$.
Moreover $S\cap (A\setminus  H')\subseteq\{p_1,p_2,p_3\}$ hence $\w_A(Z)\le m_{n+3}+\cdots +m_5+m_3+m_2$.
By Lemma \ref{claim 1} with $\eta =\alpha$, we have $\w_A(Z) \le \frac{n+1}{2}\alpha -\frac{n-1}{2}$. Therefore we conclude that  $\w_A({W'}) \le \frac{n+1}{2}\alpha -3\frac{n-1}{2} \le (n-1)(\alpha -1)+1$. 

\end{proof}

\section{The bound for the regularity index from \cite{bdp1}}\label{new bound}

Brambilla, Dumitrescu and Postinghel \cite{bdp1} gave a bound on the sum of the multiplicities for a linear system interpolating points in general position in $\PP^n$ to be only linearly obstructed. 

The notion of \emph{general position} adopted in this paper is given by the
 following condition. 
Let $(\mathbb{P}^n)^{[s]}$ be the Hilbert scheme parametrizing $s$  points of 
$\mathbb{P}^n$, and let $\mathcal{S}$ denote the point in $(\mathbb{P}^n)^{[s]}$ 
corresponding to a set $S$ of $s$ distinct points in $\mathbb{P}^n$.
The set $S\subset\PP^n$ is in general position if $\mathcal{S}$ belongs to a Zariski open subset of
$(\mathbb{P}^n)^{[s]}$.

In particular, a set of points $S\subset\PP^n$ in general position
is in linearly general position, 
it does not contain more than $n+3$ points that lie on a rational normal
curve of degree $n$, etc.

Let $\ls=\ls_{n,d}(m_1,\ldots,m_s)$ be the linear system of 
hypersurfaces of degree $d$ in $\PP^n$ passing through a collection of $s$ points in 
 general position with multiplicities at least
$m_1,\ldots,m_s$.

\begin{definition}
The  {\em (affine) virtual dimension} of $\ls$ is defined by
$$\vdim(\ls)=\binom{n+d}{n}-\sum_{i=1}^s\binom{n+m_i-1}{n}$$
and the {\em expected dimension} of $\ls$ is $\edim(\ls)=\max(\vdim(\ls),0)$.
If $\dim(\ls)=\edim(\ls)$, or equivalently then $\ls$ is said to be \emph{non-special}.
\end{definition}
\begin{remark}	\label{virtual}
Notice that $\vdim(\ls)=h^0(\PP^n,\ls)-h^1(\PP^n,\ls)$, hence $\ls$ is non-special if and only if $h^1(\PP^n,\ls)=0$.
\end{remark}

\begin{definition}[{\cite[Definition 3.2]{bdp1}}]\label{new-definition}
For any integer 
$-1\le r\le  s-1$ and for any multi-index $I(r)=\{i_1,\ldots,i_{r+1}\}\subseteq\{1,\ldots,s\}$, define the integer 
\begin{equation}\label{mult k}
k_{I(r)}:=\max(m_{i_1}+\cdots+m_{i_{r+1}}-rd,0).
\end{equation}

The {\em (affine) linear virtual dimension} of $\ls$   is the number
\begin{equation}\label{linvirtdim}
\sum_{r=-1}^{s-1}\sum_{I(r)\subseteq \{1,\ldots,s\}} (-1)^{r+1}\binom{n+k_{I(r)}-r-1}{n}.
\end{equation}
 where we set $I(-1)=\emptyset$. 
The {\em (affine) linear expected dimension} of $\ls$,  denoted by $\ldim(\ls)$, 
is defined as follows: it is $0$ if $\ls$ is contained in a linear system whose 
linear virtual dimension is non-positive, otherwise it is the maximum between 
the linear virtual dimension of $\ls$ and $0$.
If $\dim(\ls)=\ldim(\ls)$, then $\ls$ 
is said to be  {\em only linearly obstructed}. 
\end{definition}

Asking whether the dimension of a
 given linear system equals its linear expected dimension can be thought 
as a refinement of the classical question of asking whether the dimension 
equals the expected dimension.

\begin{theorem}{\cite[Theorem 5.3]{bdp1}}\label{theorem bdp}
Set $s\ge n+3$, $n\ge1$, $d\ge2$ and $d\ge m_1\ge\dots\ge m_s\ge1$.
Let $\ls=\ls_{n,d}(m_1,\dots,m_s)$ be a linear system with points in  general position.
Let $s(d)$ be the number of multiplicities equal to $d$, namely the smallest integer such that $m_{s(d)+1}<d$.
Assume that
\begin{equation}\label{b bound}
\sum_{i=1}^sm_i\le nd+b,
\end{equation}
where $b=b(\ls):=\min\{n-s(d),s-n-2\}$. 
Then $\ls$ is only linearly obstructed, i.e. $\dim(\ls)=\ldim(\ls)$. 
\end{theorem}

As an easy consequence of the above result, one obtains the following.

\begin{corollary}\label{bound for non-special}
If the same hypotheses as in Theorem  \ref{theorem bdp} are satisfied and moreover 
\begin{equation}\label{no linear obstructions}
d\ge m_1+m_2-1,
\end{equation}
then $\ls$ is non-special, namely $\dim(\ls)=\edim(\ls)$.
\end{corollary}
\begin{proof}
By Theorem \ref{theorem bdp}, $\ls$ is only linear obstructed. In fact $\dim(\ls)=\ldim(\ls)=\vdim(\ls)$,  hence $h^1(\ls)=0$ by Remark \ref{virtual}. Indeed \eqref{no linear obstructions} 
implies that the line spanned by the first two points is contained at most simply in the base locus of $\ls$ hence it does not create speciality.  
Because $m_1\ge\dots\ge m_s$, the same is true for all other lines and for all higher dimensional cycles spanned by subsets of $Z_{\textrm{red}}$. 
\end{proof}

We can rephrase the above results and give an upper bound for the regularity
 index of a collection of fat points in general position in $\PP^n$. For a linear system $\ls$, let us define the positive integer
\begin{equation}\label{integer c}
c=c(\ls):=\min\{n,s-n-2\}.\end{equation}

\begin{theorem}\label{proposition bdp bound}
Let $Z$ be a collection of fat points in  general position in $\PP^n$ with multiplicities $m_1\ge\cdots\ge m_s\ge1$. Then
\begin{equation}\label{bdp bound}
\reg(Z)\le \max\left\{m_1+m_2-1,\left\lceil \frac{\w(Z)-c}{n}\right\rceil \right\}.
\end{equation}
\end{theorem}
\begin{proof}
If $d$ is bigger or equals the number on the right hand side of \eqref{bdp bound}, then the linear system $\ls_{n,d}(m_1,\dots,m_s)$ is non-special, by Corollary \ref{bound for non-special}.
\end{proof}

\begin{remark}
When 
$m_1+m_2-1\ge\left\lceil (\w(Z)-c)/{n}\right\rceil$
the bound in Theorem \ref{proposition bdp bound} is sharp. Indeed
if $d< m_1+m_2-1$, then the corresponding linear system $\mathcal{L}$ has $h^1(\mathcal{L})>0$, because it
contains the line spanned by the first two points with multiplicity at least two 
in its base locus. This implies that $\reg(Z)\ge m_1+m_2-1$.
\end{remark}

\medskip

In the rest of the section, we will make a comparison between the bounds 
\eqref{segre bound Pn} and  \eqref{bdp bound} in the case of points in general position.

\subsection{Case $s=n+3$.}
One can easily check that if $s=n+3$, the bound \eqref{bdp bound} coincides with Segre's bound \eqref{segre bound Pn}. To see this, let us denote by $\mu,\lambda$ the integers such that 
$\w(Z)=\mu n+\lambda$, with $0\le\lambda \le n-1$. Since in this case $c=1$ (with
$c$  defined as in \eqref{integer c}), one can easily check that 
$\left\lceil (\w(Z)-c)/{n}\right\rceil=\left\lfloor (\w(Z)+n-2)/{n}\right\rfloor$ equals $\mu$ when $\lambda\le1$ and it equals $\mu+1$ when $\lambda\ge2$.

Since $s=n+3$ always lie on a rational normal curve of degree $n$ in $\PP^n$,  Proposition \ref{proposition bdp bound} provides a different proof of (\cite[Proposition 7]{ctv}) in this case.

\subsection{Quasi-homogeneous case.} For a quasi-homogeneous scheme $Z$, 
containing $s-1$ simple points 
and one fat point of weight $d$ in general position, it is easy to see that its regularity index, $\reg(Z)$, 
equals $d$ as long as $s\leq \binom {n-1+d} {d}$, obviously improving 
the bound \eqref{bdp bound}. Indeed, it follows easily that a linear system of the form 
$\ls_{n,d}(d, 1^{s-1})$ has the same dimension  as the linear system $\ls_{n-1,d}(1^{s-1})$.
The bound obtained in  Theorem \ref{proposition bdp bound} for the regularity index 
is $d$.
This improves the bound \eqref{segre bound Pn} 
as soon as $s\geq nd-d+3$. We leave it to the reader to check the details.

\subsection{Case $s\ge n+4$}
Table \ref{table 1} and Table \ref{table 2} contain the comparison between the second terms of the bounds 
\eqref{segre bound Pn} and  \eqref{bdp bound}.

\begin{table}[h]
\begin{center}
\begin{tabular}{|c|c|c|c|c|c|}
\hline
\ & $\lambda=0$ & $\lambda=0$ & $\lambda=1$&
$2\le\lambda\le s-n-2$& $s-n-2<\lambda\le n-1$\\
\ & $s=2n+2$  & $s\le 2n+1$ &\ &
\ & \ \\
\hline
\eqref{segre bound Pn} & $\mu$ & $\mu$ & $\mu$ & $\mu+1$ & $\mu+1$\\ 
\eqref{bdp bound} & $\mu-1$ & $\mu$ &  $\mu$ &$\mu$ & $\mu+1$\\  
\hline 
\end{tabular}\medskip
\end{center}\caption{Comparison table in the case $n+4\le s\le 2n+2$. }\label{table 1}
\end{table}
\begin{table}[h]
\begin{center}
\begin{tabular}{|c|c|c|c|}
\hline
\ & $\lambda=0$  & $\lambda=1$ & $2\le \lambda\le n-1$\\
\hline
\eqref{segre bound Pn} & $\mu$ & $\mu$ & $\mu+1$\\ 
\eqref{bdp bound} & $\mu-1$ & $\mu$ & $\mu$\\ 
\hline
\end{tabular}\medskip
\end{center}\caption{Comparison table in the case $s\ge 2n+3$. }\label{table 2}
\end{table}

\medskip

We conclude this section with a list of examples  in which  Theorem \ref{proposition bdp bound} provides an improvement of the Segre's bound. 
The bounds we present in the examples below are sharp, in other words, the regularity index for the corresponding scheme of fat points is given by \eqref{bdp bound}.

\begin{example}
For the planar case $n=2$, the bound \eqref{bdp bound} improves the bound given by 
Segre \eqref{segre bound P2}. One can easily check this by considering the scheme
 given by six double points for which Segre's bound \eqref{segre bound P2} equals $6$.
 However the linear system of quintic curves $\ls=\ls_{2,5}(2^6)$ has vanishing
 cohomology 
group $H^1(\PP^2,\ls)$, as predicted by our bound \eqref{bdp bound} that is $5$. 
\end{example} 

\begin{example} For a scheme of nine double points in $\PP^{3}$,
Segre's bound \eqref{segre bound Pn} is $6$, but $\ls=\ls_{3,5}(2^9)$ has vanishing 
$H^1(\PP^3, \ls)$, as predicted by the bound \eqref{bdp bound} that is $5$. 
\end{example}

\section{Modification of Segre's conjecture for arbitrary number of points}\label{Conjecture}
In this section we introduce the following conjecture for points in arbitrary position and present the evidences we have for it.

\begin{conjecture}\label{Seg}
Fix integers $d\ge 2$, $n> 2$.  Let $S\subset \PP^n$ be a finite 
collection of $s$ points. Fix integers $m_1,\dots, m_s$ and set 
$Z:= \sum_{i=1}^s  m_ip_i$. Then $h^1(\Ii _Z(d)) =0$ if all the following conditions are satisfied:
\begin{enumerate}
\item $\w(Z) \le nd+1$;
\item $\w_L(Z) \le d+1$ for each line $L\subset \PP^n$;
\item for all integers $r=2,\dots ,n-1$ and every $r$-dimensional linear subspace $L\subset \PP^n$ we have $\w_L(Z)
\le rd+2$; if $\w_L(Z) =rd+2$ we also assume $h^1(L,\Ii _{Z\cap L}(d)) =0$. 
\end{enumerate}
\end{conjecture}

Condition (1) (resp. (2)) of Conjecture \ref{Seg} is to make sure that no rational normal curve 
(resp. line), spanned by points of S, is contained in the singular locus of
 $\mathcal{I}_Z(d)$. See also Remark 4.3 below. 
Moreover condition (3) says that if neither does any  higher dimensional 
linear spaces $L$ spanned by points of $S$ (namely $\w_L(Z)\le d+1$), then 
$h^1(\mathcal{I}_Z(d))=0$.
When $\w_L(Z)=d+2$, condition (3) says that if both (1) and (2) are satisfied,
 the non-vanishing $h^1(\mathcal{I}_Z(d))>0$ happens because of  the non-vanishing of the
 first cohomology of the same sheaf restricted to $L$.

\begin{lemma}\label{ss1}
 Assume the existence of a closed subscheme $W\subsetneq Z$
such that $h^1(\Ii _W(d)) >0$. Then $h^1(\Ii _Z(d)) >0$.\end{lemma}

\begin{proof}
Since $\dim (Z) =0$, we have $h^1(Z,\Ii _{W,Z}(d)) =0$. Hence
the restriction map $H^0(\mathcal {O}_Z(d)) \to H^0(\mathcal {O}_W(d))$ is surjective. Since $h^1(\Ii _W(d)) >0$, we get
$h^1(\Ii _Z(d)) >0$.
\end{proof}

\begin{remark}\label{ss12}
In the same notation as in Conjecture \ref{Seg},
assume the existence of a closed subscheme $T\subset \PP^n$
such that $h^0(T, \mathcal {O}_T(d)) < \w_T(Z)$. Then $h^1(T,\Ii _{Z\cap T}(d)) >0$. 
Hence $h^1(\Ii _{Z\cap T}(d)) >0$. Lemma \ref{ss1} gives $h^1(\Ii _Z(d))>0$.
In particular to have any chance that $h^1(\Ii _Z(d))=0$, we need $\w_C(Z) \le md+1$ for every rational normal curve of degree $m$ of an $m$-dimensional
linear subspace of $\PP^n$ (we allow the case $m=n$). In particular condition (2) of Conjecture \ref{Seg} is a necessary condition.
\end{remark}

\begin{theorem}\label{cc1}
Fix integers $n\ge 2$, $d\ge 4$, $d\ge m_1+2$, $s\ge 2n+3$, 
$m_1\ge \cdots \ge m_s > 0$, $\sum _i m_i =nd+2$, $m_1+m_{2n+2}\le d$ and $m_1+m_2 \le d+1$.
Let $S = \{p_1,\dots ,p_s\} \subset \PP^n$ be a finite subset of points in linearly general position. Set $Z:= \sum _i m_ip_i$.
We have $h^1(\Ii _Z(d)) =0$ if and only if $S$ is not contained in a rational normal curve of degree $n$ of $\PP^n$.
\end{theorem}

\begin{proof}
If $S$ is contained in a rational normal curve of $\PP^n$, then $h^1(\Ii _Z(d)) >0$, see \cite[Proposition 7]{ctv}. We use induction on $d$ starting from the case
$d =\max \{4,m_1+2\}$. We check in cases (d), (e) and (f) the starting cases of the induction. In cases (a) and (b) we do not use the inductive assumption, but reduce
the proof to a game with a scheme $W$ with $h^1(\Ii _W(d-2)) =0$ by \cite[Theorem 6]{ctv}.

Assume $h^1(\Ii _Z(d)) >0$. For each $p\in S$ let $m_p$ be the multiplicity of $p$ in $Z$. Let $B_1$ (resp. $B'_1$) be the set of all lines $L\subset \PP^n$ such that $\w_L(Z) =d+1$
(resp. $\w_L(Z)=d$). Since $S$ is in linearly general position, $L\in B_1$ (resp. $B'_1$) if and only if $L$ is a line spanned by two different points, say $p$ and $q$,
with $m_p+m_q = d+1$ (resp. $m_p+m_q=d$). In particular if $B_1\ne \emptyset$, then $m_1+m_2=d+1$. 
If $m_1>m_2$, then the elements of $B_1$ are the lines spanned by $p_1$ and one the points $p_i$, $2\le i \le a$, where $a$ is the maximal integer $i\in \{2,\dots ,s\}$ with $m_i=m_2$.
If $m_1=m_2$, then either $B_1=\emptyset$ (case $d$ even) or $B'_1 =\emptyset$ (case $d$ odd) and $B_1\cup B'_1\ne \emptyset$ if and only if $m_1=m_2 =\lceil d/2\rceil$. In this case
the elements of $B_1\cup B'_1$ are the lines spanned by two points of $\{1,\dots ,a'\}$, where $a'$ is the maximal integer $i\le s$ with $m_i=m_1=m_2$.
We consider the following cases.

\medskip

\quad Case (a):  $m_1=m_2$ and $d$ odd. 

Let $Q\subset \PP^n$ be a general quadric hypersurface containing $\{p_1,\dots, p_{2n+1}\}$. Such a quadric exists, because
$2n+1 < \binom{n+2}{2}$. Set $W:= \mathrm{Res}_Q(Z)$ and $S':= W_{\mathrm{red}}$. For each $p\in S$ let $u_p$ be the multiplicity of $p$ in $W$. If $p= p_i$, set
$u_i:= u_{p_i}$.
We have $u_p =m_p$ if $p\notin Q$, $u_p=m_p-1$ if $p$ is a smooth point of $Q$ and $u_p =\max \{0,m_p-2\}$ if $p$ is a singular point of $Q$. Since $S'\subseteq S$,
$S'$ is in linearly general position. We assume $\sharp (S')\ge n+1$, i.e. that $S'$ spans $\PP^n$; see the proof of case (e) for the case $\sharp (S')\le n$. Since $\sharp (S\cap Q) \ge 2n+1$, we have $\sum _{p\in S'} u_p \le nd+2-2n-1 \le n(d-2)+1$. If $B_1\cup B'_1 =\emptyset$,
then $\w_L(W) \le d-1$ for all lines $L\subset \PP^n$ and hence $h^1(\Ii _W(d-2)) =0$ (\cite[Theorem 6]{ctv}). Now assume $B_1\cup B'_1\ne \emptyset$
and let $e$ be the maximal integer $i\le s$ with $m_i = \lceil d/2\rceil$. Since $\sum _i m_i=nd+2 < (2n+1)\lceil d/2\rceil$, we have $e\le 2n$. Hence if $L\in B_1\cup B'_1$,
then $\w_L(W)\le \w_L(Z)-2 \le d-1$ and hence  $h^1(\Ii _W(d-2)) =0$ (\cite[Theorem 6]{ctv}). Since $h^1(\Ii _Z(d)) >0$, the Castelnuovo residual sequence with respect to $Q$
gives $h^1(Q,\Ii _{Z\cap Q}(d)) >0$. Since the zero-dimensional scheme $Z\cap Q$ is contained in
the scheme $Z':= \sum _{p\in S\cap Q} m_pp$, we get $h^1(\Ii _{Z'}(d)) >0$ (Lemma \ref{ss1}). The support of $Z'$ is in linearly general position and $\w (Z'\cap L)\le \w (Z\cap L)\le d+1$
for all lines $L$. Hence by \cite[Theorem 6]{ctv} we have $\sum _{p\in S\cap Q} m_p\ge nd+2$, i.e. $S\cap Q =S$, i.e. $S$ is in the base locus of $|\Ii _{\{p_1,\dots ,p_{2n+1}\}}(2)|$,
i.e. $|\Ii _{\{p_1,\dots ,p_{2n+1}\}}(2)| =|\Ii _S(2)|$ and $h^0(\Ii _S(2)) =\binom{n+2}{2}-2n-1$.
Since $S$ is in linearly general position,  $h^1(\Ii _A(2)) =0$ for all $A\subset S$ with
$\sharp (A) \le 2n+1$. By \cite[Lemma 3.9]{he}, $S$ is contained in a rational normal curve.

\medskip

\quad Case (b):  $m_1=m_2$ and $d$ even. 

Notice that in this case $B_1=\emptyset$. If $B'_1=\emptyset$, i.e. if $m_1<d/2$, then we make the same construction as in Case (a) with $e=0$.
We still have $\w_L(W) \le d-1$ for all lines $L\subset \PP^n$, because $\w_L(Z)\le d-1$ for every line $L$ and so $h^1(\Ii _W(d-2)) =0$. Now
assume $m_1=d/2$. Let $f$ be the maximal integer $i\le s$ such that $m_i = d/2$. Since $s\ge 2n+3$ and $1+ (2n+2)d/2 > nd +2$, we have $f\le 2n+1$. We may repeat the proof of Case (a), because
$\w_L(W) \le d-1$ holds. 

\medskip

\quad Case (c):  $m_1>m_2$ and $d>4$.

 If $m_1+m_2 \le d-1$ and $d\ge 6$, then we may repeat the proof of Case (a), because $\w_L(W)\le d-1$
for all lines $L$ (see the details in case (e)). Hence we may assume $m_1+m_2\ge d$. Let $x$ be the maximal integer $i\le s$ such that $m_i=m_2 =d+1-m_1$ (with
the convention $x=0$ if there is no such an integer, i.e. if $m_2=d-m_1$) and let $g$ be the maximal integer $i\le s$
such that $m_i=d-m_1$ with the convention $g =x$ if there is no such an integer. If $g\le 2n+1$, then we may repeat the proof of Case (a), because $\w_L(W)\le d-1$
for all lines $L$. Hence we may assume $g\ge 2n+2$.  Let $H$ be the hyperplane spanned by
$p_1,\dots ,p_n$. Set $U:= \mathrm{Res}_H(Z)$ and let $S_1:= U_{\mathrm{red}}$ be the support of $U$. For each $p_i\in S$ let $r_i$ be the multiplicity of $p_i$ in $U$.
We have $r_i=m_i-1$ if $i\le n$ and $r_i=m_i$ if $i>n$. Since each element of $B_1$ contains $p_1$, we have $\w_L(U)\le d$ for all lines $L$. We have $\sum _i r_i =n(d-1)+2$.  
Since $d\geq m_1+2$ and $g\geq 2n+2$, then $m_i\geq 2$ for
 all $i\leq g$, in particular for all $i\leq n$. Hence $S_1=S$ and $\sharp (S_1)\ge 2n+3$. We order the sequence $r_1,\dots ,r_s$ in non-decreasing order $y_1\ge \cdots \ge y_s$. We have $y_1=r_1=m_1-1$. Since $d\ge m_1+2$, $m_1>m_2$, $r_1=m_1-1$
and $y_i\le m_i$ for all $i$,
we have $d-1\ge y_1+2$. We have $y_1 +y_{2n+2} \le d-1$, because $y_{2n+2} \le m_{2n+2}$ and
$y_1 =m_1-1$. Hence by the inductive assumption there is a rational normal curve $C\supset S_1 =S$.

\medskip 

\quad Case (d): Assume
$d=4$ and so $m_1 \le 2$. Let $f$ be the maximal integer $i$ such that $m_i =2$ with the convention $f=0$ if $m_1=1$. Since $\w (Z) =4n+2$, we have $f\le 2n+1$. Let $Q$ be a general
quadric hypersurface containing $\{p_1,\dots ,p_{2n+1}\}$. Set $W:= \mathrm{Res} _Q(Z)$. For each $p\in S$ set $m'_p= m_p$ if $p\notin Q$ and $m'_p =m_p-1$ if $p\in Q$.
Set $Z' = \sum _{p\in S} m'_pp$. Since $f\le 2n+1$, $W$ is a reduced scheme with cardinality $\le 2n+1$ and in linearly general position. Therefore
$h^1(\Ii _W(2)) =0$. The Castelnuovo's sequence of $Q$ gives $h^1(Q,\Ii _{Z\cap Q}(4)) >0$. Since $Z'\cap Q =Z\cap Q$, Lemma \ref{ss1} gives
$h^1(\Ii _{Z'}(4)) >0$. Since $Z'_{\mathrm{red}} \subseteq S$ is in linearly general position, \cite[Theorem 6]{ctv} gives $\w (Z') >4n+1$, i.e. $Z' =Z$.
Hence $Q\supset S$. Since a general quadric hypersurface containing $\{p_1,\dots ,p_{2n+1}\}$ contains $S$, $S$ is in linearly general position and
$\sharp (S)\ge 2n+3$, then $S$ is contained in a rational normal curve (\cite[Lemma 3.9]{he}).  

\medskip 

\quad Case (e) Assume $d=5$ and hence $m_1\le 3$. Let $h$ be the maximal integer $i$ such that $m_i \ge 3$. Since $\w (Z)=5d+2$, we have $h\le 2n+1$. Take $Q$ and $W =\mathrm{Res} _Q(Z)$ as in Case (e). We have
$\w(W) \le 3n+1$ and $W_{\mathrm{red}}$ is in linearly general position. Call $u_1 \ge u_2 \ge \cdots $ the multiplicities in $W$ of the points
of $W_{\mathrm{red}}$. Since $h\le 2n+1$, we $u_1\le 2$ and hence $u_1+u_2 \le 4$. If $W_{\mathrm{red}}$ spans $\PP^n$,
then \cite[Theorem 6]{ctv} gives $h^1(\Ii _W(3)) =0$ and we conclude as in Case (d). Now assume that $W_{\mathrm{red}}$ spans a linear subspace $M$ of dimension $r<n$.
Write $W_{\mathrm{red}} =\{q_1,\dots ,q_{r+1}\}$ with $q_i$ appearing with multiplicity $u_i$ in $W$. 
By \cite[Theorem 6]{ctv} we have $h^1(M,\Ii _{W\cap M}(3)) =0$. Take a linear subspace $M\subset N\subseteq \PP^n$ with $\dim (N)=r+1$. See
$M$ as a hyperplane of $N$ to compute the residual scheme $\mathrm{Res} _M(W\cap N)$ of the scheme $W\cap N\subset N$.
Each $q_i$ occurs in $\mathrm{Res} _N(W\cap N)$ with multiplicity $u_i-1\le 1$. Hence $h^1(N,\Ii _{\mathrm{Res} _M(N\cap Z)}(1)) =0$ and so
$h^1(N,\Ii _{\mathrm{Res} _M(N\cap Z)}(2)) =0$. 
The Castelnuovo's sequence of $M$ gives $h^1(N,\Ii _{N\cap W}(3)) =0$. If $N=\PP^n$, 
then $h^1(\Ii _W(3)) =0$.
If $r+1 <n$ we take a flag of linear subspaces $N\subset N_1\subset \cdots \subset N_{n-r-1}=\PP^n$ with $\dim (N_i) =r+1-i$ for all $i$. After $n-r-1$ steps
we get $h^1(\Ii _W(3))=0$. Then we conclude as in Case (d)

\medskip 

\quad Case (f) Assume $d=m_1+2 \geq 6$. Since $m_1+m_2\le d+1$, we have $m_2\le 3$. Let $g$ be the maximal integer $i$ such that $m_g \ge 3$. Since
$m _{2n+2} +m_1\le d$, we have $g\le 2n+1$. Take $Q$ and $W$ as in Case (e). As in Case (a) or (d) it is sufficient to prove that $h^1(\Ii _W(d-2)) =0$.
Call $u_1\ge u_2\ge\cdots \ge $ the multiplicities in $W$ of the points of $W_{\mathrm{red}}$. Since $g\le 2n+1$, we have $u_2\le 2$. Since $u_1 \le m_1-1 \le d-3$
and $W_{\mathrm{red}}$ is in linearly general position, we have $h^1(\Ii _W(d-2)) =0$ and we conclude as in the last two cases.
\end{proof}

\subsection{Conjecture \ref{Seg} holds for $n=3$}

In this section we prove that Conjecture \ref{Seg} holds for $n=3$.

\begin{proposition}\label{ccc12}
Conjecture \ref{Seg} is true if $n=3$.
\end{proposition}

\begin{proof}
Let $d$ be an integer and $Z$ a fat point scheme for which  conditions (1), (2) and (3) 
of Conjecture \ref{Seg} are satisfied, but $h^1(\mathcal{I}_Z(d))>0$.
Set $S := Z_{\mathrm{red}}$.  By \cite{fl} or \cite{t} Segre's conjecture holds in $\PP^3$. Hence we may assume the existence of a plane $H\subset \PP^3$ such that
$\w_H(Z) =2d+2$.  By assumption we have $h^1(H,\Ii _{Z\cap H,H}(d)) =0$.
 Hence the residual sequence of $H$ and $Z$ gives $h^1(\Ii _{\mathrm{Res}_H(Z)}(d-1)) >0$.
Since $\w _L(Z) \le d+1$ for every line and $\w_H(Z) =2d+2$, 
we have $\sharp (H\cap S) \ge 4$. Hence $\w (\mathrm{Res}_H(Z)) \le 3(d-1)+1$.
Fix a line $L\subset \PP^3$. Since $\w _L(Z)\le d+1$, we have 
$\w _L(\mathrm{Res}_H(Z))\le d+1$ with equality if and only if $\w _L(Z) = d+1$ and
 $S\cap L\cap H=\emptyset$.
If $\w _L(Z) = d+1$ and $S\cap L\cap H=\emptyset$, then 
$\w (Z)\ge \w_H(Z)+\w_L(Z)=3d+3$, a contradiction. 
Therefore $\mathrm{Res}_H(Z)$ satisfies the Segre condition 
(Conjecture \ref{conjecture gen segre bound}) with respect
to lines, i.e. $\w_L(\mathrm{Res}_H(Z))\le d$ for all lines $L\subset\PP^3$.
 Since Segre condition is true even in degree $1$ by \cite{fl} or \cite{t}, 
the fact that $h^1(\Ii _{\mathrm{Res}_H(Z)}(d-1)) >0$
implies  the existence of a plane $M\subset \PP^2$ such
that $\w _M(\mathrm{Res}_H(Z))  \ge 2(d-1)+1=2d-1$.

If $\sharp (S\cap H\cap M)\ge 4$, then we get 
$\w _M(Z) \ge\w _M(\mathrm{Res}_H(Z))+4\ge 2d+3$, which is in contradiction with 
assumption (3). 
Therefore $\sharp (S\cap H\cap M)\le 3$. 
Notice that in this case, since $\sharp (S\cap H)\ge 4$, we have $M\ne H$. Since $M\cap H$ is a line, we have
$\w_{H\cap M}(\mathrm{Res}_H(Z)) \le d$. We have
$3d+1 \ge \w (Z) \ge \w _H(Z) +\w _M(\mathrm{Res}_H(Z)) -\w _{H\cap M}(\mathrm{Res}_H(Z))\ge
2d+2+2d-d.$
This gives a contradiction.

\end{proof}

For a specific fat point scheme $Z$, a good starting point in order to understand
 the value of $h^1(\Ii _Z(d))$ is to consider  linear subspaces $V$ such that  $h^1(V,\Ii _{Z\cap V,V}(d)) =0$
and $\w _V(Z) \gg \dim (V)\cdot d$. 
Then one should look at the residual exact sequences with respect 
to hyperplanes or hyperquadrics $T$ containing $V$ and for which 
$\w_T(Z)$ is large.

\subsection{Achieving Segre's bound}
A finite subset $A\subset \PP^n$ is said to be in \emph{uniform position} or that it has the \emph{uniform position property} if any two subsets of $A$ with the
the same cardinality have the same Hilbert function. We recall that \cite[Theorem 6]{ctv} proves the Segre's  bound \eqref{segre bound Pn} for points in linearly general position 
$$\reg(Z)\leq \max\left\{m_1+m_2-1,\left\lfloor \frac{\w(Z)+n-2}{n}\right\rfloor\right\}.$$

Moreover, \cite[Proposition 7]{ctv} shows that equality holds if the set $\{p_1,\dots ,p_s\}$ is contained in a rational normal curve of $\mathbb {P}^n$. In \cite[Problem 1]{ctv}, the authors asked if the equality in \eqref{segre bound Pn} implies that $\{p_1,\dots ,p_s\}$ is contained in a rational normal curve of $\mathbb {P}^n$.
 It was shown to be true in \cite{c} and \cite[Theorem 2.1]{tv} for points in uniform position (at least if $m_{2n+3}\ge n$). We will show that this is not always the case by exhibiting the following family of examples.

\begin{example}\label{a1}
Fix positive integers $n\ge 4$, $m_i>0$ and $d\ge 4$ such that $m_1+\cdots +m_s = nd + \alpha$ with $1 \le \alpha \le n-1$.
Assume the existence of an integer $b\in \{s-\alpha, \dots ,s-1\}$ such that
$\sum _{i=b+1}^{s} m_i \le \alpha$; for instance, take $b=n-\alpha$ and $m_i=1$ for all $i>b$. Let $C\subset \mathbb {P}^n$ be a rational normal curve of degree $n$.
Take $p_i\in C$, $1\le i \le b$, distinct and $p_j\in \mathbb {P}^n\setminus C$,
$b+1 \le j \le s$, with the only restriction that the set $\{p_1,\dots ,p_s\}$ is in linearly general position (e.g. we take $p_j$ general for $j>b$).
Set $Z':= \sum _{i=1}^{b} m_ip_i$. By \cite[Proposition 7]{ctv}, we have
$\reg(Z') =t$, where $t$ is the Segre's bound. By \cite[Proposition 5]{ctv}, we have $\reg(Z) \le t$. Lemma \ref{ss1} implies that $\reg(Z') \le \reg(Z)$.
Hence $\reg(Z)=t$. If $s-b \ge n+3$, then $C$ is the only rational normal curve containing
$p_1,\dots ,p_b$. Hence $\{p_1,\dots ,p_s\}$ is contained in no rational normal curve.
\end{example}

Fix any positive integer $x$. We say that $A$ is in \emph{uniform position in degree $\le x$} if for any $E, F\subset A$
with  $\sharp (E) =\sharp (F)$ we have $h_E(t) =h_F(t)$ for $t=1,\dots ,x$. Uniform position in degree $\le 1$ is equivalent to linearly general position.

The proof of \cite[Theorem 2.1]{tv} works verbatim just assuming that the set has  uniform position in degree $\le 2$. Trung and Valla gave another result in which their conjecture
is true and with the set only assumed to be in linearly general position (\cite[Theorem 1.6]{tv}). So we do not see a natural way to improve Theorem \ref{cc1} to the case
$\w(Z) >nd+2$, nor to weaken the assumption $m_1+m_{2n+2}\le d$.

In the following example the set $S$ is in uniform position.

\begin{example}
Fix integers $d\ge 2$ and $n\ge 3$ and set $s:= (n-1)d+3$. Fix a point $p\in \mathbb {P}^n$, a hyperplane $H\subset \mathbb {P}^n$ such that $p\notin H$
and a rational normal curve $C\subset H$ of degree $n-1$. 
Let $T\subset \mathbb {P}^n$ be the cone with vertex $p$ over $C$. Fix a general
$S'\subset T$ with $\sharp (S') =s-1$ and set $S:= \{p\}\cup S'$. Set $m_p:= d$ and 
$m_q=1$ for all $q\in S'$. Write $Z:= \sum _{q\in S} m_qq$.
Let $\ell : \mathbb {P}^n\setminus \{p\} \to H$ denote the linear projection from $p$ onto $H$. 
Set $A:= \ell (S')$. We have $\sharp (A) = (n-1)d+2$. Since $A$
is contained in a rational normal curve of $H$, we have $h^1(H,\mathcal {I} _A(d)) =1$. The linear system $|\mathcal {I} _{dp}(d)|$ is the set of all degree $d$ cones
with vertex containing $p$. Hence $h^0(\mathcal {I} _Z(d)) = h^0(H,\mathcal {I} _A(d))$ and
 $h^1(\mathcal {I} _Z(d)) =1$. Since $s \ge n+4$ and $S'$ is general in $T$,
$S$ is not contained in a rational normal curve of degree $n$ of $\mathbb {P}^n$. 
We claim that $S$ is in uniform position. To see this, fix an integer $k\ge 0$. Since $T$ is an irreducible
variety and $\mathcal {O} _{\mathbb {P}^n}(k)$ has no base points, for a general finite set $E\subset T$ we
have $h^0(\mathcal {I} _E(k)) = \max \{h^0(\mathcal {I} _T(k)), \binom{n+k}{n}-\sharp (E)\}$ and $h^0(\mathcal {I} _{E\cup \{p\}}(k)) = \max \{h^0(\mathcal {I} _T(k)) , \binom{n+k}{n}-\sharp (E)-1\}$. Since $S'$ is general in $T$, any two subsets of $S$ with the same cardinality have the same Hilbert function.
\end{example}

\begin{appendix}

\section{Macaulay 2 Codes}
\label{app}
\setcounter{section}{1}
In this section we provide Macaulay 2 \cite{macaulay} scripts for the computations  that cover the initial cases of the proofs of the results contained in Section \ref{generalised segre's bound}.

\subsection{Code 1}\label{app1}

\begin{verbatim}
KK=ZZ/32749;
R=KK[e_0..e_4]

d=5; 
N=binomial(d+4,4);

f=ideal(e_0..e_4);
fd=f^d;
T=gens gb(fd)
J=jacobian(T); 
JJ=jacobian(J); 

p0=matrix{{e_0^0,0,0,0,0}}; 
p1=matrix{{0,e_0^0,0,0,0}}; 
p2=matrix{{0,0,e_0^0,0,0}}; 
p3=matrix{{0,0,0,e_0^0,0}};
p4=matrix{{e_0^0,e_0^0,e_0^0,e_0^0,0}};  
p5=matrix{{0,0,0,0,e_0^0}}; 
p6=matrix{{e_0^0,e_0^0,-e_0^0,0,e_0^0}};  

mat=random(R^1,R^N)*0;
mat=(mat||sub(JJ,p0));
mat=(mat||sub(JJ,p1));
mat=(mat||sub(JJ,p2));
mat=(mat||sub(JJ,p3));
mat=(mat||sub(JJ,p4));
mat=(mat||sub(JJ,p5));
mat=(mat||sub(JJ,p6));

r=rank mat;
exprank=7*binomial(6,4); 

print(r,exprank) 
\end{verbatim}

\subsection{Code 2}\label{app2}
\begin{verbatim}
KK=ZZ/32749;
R=KK[e_0..e_5]

d=3; 
N=binomial(d+5,5);

f=ideal(e_0..e_5);
fd=f^d;
T=gens gb(fd)
J=jacobian(T);

p0=matrix{{e_0^0,0,0,0,0,0}}; 
p1=matrix{{0,e_0^0,0,0,0,0}}; 
p2=matrix{{0,0,e_0^0,0,0,0}}; 
p3=matrix{{0,0,0,e_0^0,0,0}}; 
p4=matrix{{0,0,0,0,e_0^0,0}};
p5=matrix{{0,0,0,0,0,e_0^0}};
p6=matrix{{e_0^0,e_0^0,e_0^0,e_0^0,0,0}};  
p7=matrix{{e_0^0,e_0^0,0,0,e_0^0,e_0^0}};  

mat=random(R^1,R^N)*0;
mat=(mat||sub(J,p0));
mat=(mat||sub(J,p1));
mat=(mat||sub(J,p2));
mat=(mat||sub(J,p3));
mat=(mat||sub(J,p4));
mat=(mat||sub(J,p5));
mat=(mat||sub(J,p6));
mat=(mat||sub(J,p7));

r=rank mat;
exprank=8*6; 

print(r,exprank) 
\end{verbatim}
\end{appendix}

\providecommand{\bysame}{\leavevmode\hbox to3em{\hrulefill}\thinspace}

\end{document}